\documentclass[letterpaper, 10 pt, conference, onecolumn]{ieeeconf}  %

\usepackage{graphicx} %
\usepackage{amsmath} %
\usepackage{amssymb}  %
\usepackage{cleveref}
\usepackage{xcolor}
\usepackage{algorithm} %
\usepackage{algpseudocode} %
\usepackage{cite}

\newtheorem{remark}{Remark}

\newtheorem{assumption}{Assumption}

\newtheorem{lemma}{Lemma}
\newtheorem{theorem}{Theorem}

\definecolor{verde}{RGB}{0,170,0}

\newcommand{\du}{^}
\newcommand{\ud}{_}

\newcommand{\oprocend}{\hfill\hbox{\rule{1.25ex}{1.25ex}}}

\newcommand{\domX}{\mathbf{X}}
\newcommand{\domXi}{\boldsymbol{\Xi}}
\newcommand{\domU}{\mathbf{U}}
\newcommand{\domZ}{\mathbf{Z}}

\newcommand{\cG}{\mathcal{G}}

\newcommand{\cL}{\mathcal{L}}

\newcommand{\cS}{\mathcal{S}}

\newcommand{\R}{\mathbb{R}}
\newcommand{\N}{\mathbb{N}}

\newcommand{\norm}[1]{\left\| #1 \right\|}
\newcommand{\snorm}[1]{\| #1 \|}

\newcommand{\iter}{{t}}
\newcommand{\iterp}{{\iter+1}}

\newcommand{\initer}{{\tau}}
\newcommand{\initerp}{{\initer +1}}

\newcommand{\x}{x}
\newcommand{\z}{z}
\newcommand{\uu}{u}

\newcommand{\pr}{\delta}
\newcommand{\prm}{\delta_{\text{MPC}}}

\newcommand{\xtp}{\x\ud{\iterp}}
\newcommand{\xt}{\x\ud{\iter}}

\newcommand{\xot}{\x\ud{1,\iter}}
\newcommand{\xtt}{\x\ud{2,\iter}}

\newcommand{\xtap}{\x\ud{\initerp}}
\newcommand{\xta}{\x\ud{\initer}} 

\newcommand{\xitp}{\xi\ud{\iterp}}
\newcommand{\xit}{\xi\ud{\iter}}

\newcommand{\ztp}{\z\ud{\iterp}}
\newcommand{\zt}{\z\ud{\iter}}

\newcommand{\slow}{f}
\newcommand{\fast}{g}
\newcommand{\mpc}{\mathcal{A}}

\newcommand{\ut}{\uu\ud{\iter}}

\newcommand{\uta}{\uu\ud{\initer}}

\newcommand{\cost}{\ell}
\newcommand{\costf}{\ell_{\hor}}

\newcommand{\hor}{T}

\newcommand{\red}{\slow_{R}}

\newcommand{\xieq}{\xi\ud{\text{eq}}}

\newcommand{\zstar}{\z\ud{\star}}

\newcommand{\scale}{1}

\newcommand{\lipeq}{L_{\xi}}

\newcommand{\txi}{\tilde{\xi}}
\newcommand{\txit}{\txi\ud\iter}
\newcommand{\txitp}{\txi\ud\iterp}

\newcommand{\tz}{\tilde{z}}
\newcommand{\tzt}{\tz\ud\iter}
\newcommand{\tztp}{\tz\ud\iterp}

\DeclareMathOperator{\col}{\mathrm{col}}
\DeclareMathOperator{\diag}{\mathrm{diag}}

\newcommand{\tfast}{\tilde{\fast}}
\newcommand{\tmpc}{\tilde{\mpc}}

\newcommand{\xstar}{\x\ud{\star}}

\newcommand{\Uextra}{\cG}

\def\algo/{{Suboptimal and Reduced-Order MPC}}
\def\algoCAPS/{{SUBOPTIMAL AND REDUCED-ORDER MPC}}

\IEEEoverridecommandlockouts                              %

\newtheorem{example}{Example}
\crefname{example}{example}{examples}
\Crefname{example}{Example}{Examples}

\title{\LARGE \bf Nonlinear MPC for Feedback-Interconnected Systems: \\a Suboptimal and Reduced-Order Model Approach}

\author{Stefano Di Gregorio, Guido Carnevale and Giuseppe Notarstefano%
\thanks{
    Work partially funded by FISA-2023-00210 project APACHE - CUP J53C25000520001, by the European Union - Next Generation EU - under the National Recovery and Resilience Plan (NRRP), Mission 4, Component 2, Investment 3.3. CUP J33C24001490009 and by IMA S.p.A.
}%
\thanks{The authors are with the Department of Electrical, Electronic and Information Engineering, University of Bologna, 40136, Bologna, Italy (e-mail: \{stefano.digregorio, guido.carnevale, giuseppe.notarstefano\}@unibo.it).}%
}

\begin{document}

\maketitle
\thispagestyle{empty}
\pagestyle{empty}

\begin{abstract}
   In this paper, we propose a suboptimal and reduced-order Model Predictive Control (MPC) architecture for discrete-time feedback-interconnected systems.
    The numerical MPC solver: (i) acts suboptimally, performing only a finite number of optimization iterations at each sampling instant, and (ii) relies only on a reduced-order model that neglects part of the system dynamics, either due to unmodeled effects or the presence of a low-level compensator.
   We prove that the closed-loop system resulting from the interconnection of the suboptimal and reduced-order MPC optimizer with the full-order plant has a globally exponentially stable equilibrium point. Specifically, we employ timescale separation arguments to characterize the interaction between the components of the feedback-interconnected system.
   The analysis relies on an appropriately tuned timescale parameter accounting for how fast the system dynamics are sampled.
   The theoretical results are validated through numerical simulations on a mechatronic system consisting of a pendulum actuated by a DC motor.
\end{abstract}

\section{INTRODUCTION}
\label{sec:introduction}

Model Predictive Control (MPC) is receiving increasing attention in both academia and industry as a powerful control technique for dynamical systems, see, e.g.,~\cite{bemporad2006model,kouvaritakis2016model,rawlings2020model} for a comprehensive overview of its theoretical and practical aspects.

The main drawback of MPC schemes lies in their computational complexity, which stems from the need to solve an optimal control problem online at each sampling instant.
Consequently, a central challenge in current MPC research is to mitigate this limitation.
In this context, we focus on two strategies that are widely adopted in both theoretical studies and practical implementations: (i) model order reduction, and (ii) suboptimal control approaches.
When dealing with complex systems characterized by hierarchical structures or multiple interconnected subsystems, the computational complexity of MPC can become prohibitive.
In such cases, reduced-order models can be employed to simplify the optimization problem while still capturing the essential dynamics of the system.
In this spirit, works~\cite{rosolia2020multi,csomay2022multi,rosolia2022unified} propose multi-rate control architectures, in which high-level routines based on MPC operate on reduced-order models and provide setpoints to low-level controllers acting on the full-order ones.
This paradigm, which is very common in robotic applications, has been formalized in the tutorial work~\cite{matni2024towards}.
In~\cite{loehning2014model}, the authors guarantee the stability of the full closed-loop system despite the use of a reduced model in the MPC optimizer.
A reduced-order MPC scheme is presented in~\cite{lorenzetti2019reduced}, where stability and constraint satisfaction are proven while explicitly accounting for model reduction errors, whose bounds can be computed, e.g., as in~\cite{lorenzetti2020error}.
In~\cite{kartmann2024certified}, Galerkin projection methods are employed to operate on a low-dimensional surrogate model.
The authors of~\cite{wang2022tube} present a robust MPC framework for two-time-scale linear systems, where only a reduced-order model neglecting the fast dynamics is used in the MPC optimizer.

Suboptimal MPC schemes have received considerable attention in the recent literature as a means to reduce the computational burden associated with standard MPC approaches.
The main idea behind suboptimal MPC is to iteratively compute an approximate solution to the optimal control problem, rather than solving it exactly at each sampling instant.
See, e.g.,~\cite{diehl2005real,gros2020linear} for an overview of suboptimal MPC schemes applied to nonlinear systems.
The stability properties of suboptimal MPC have been investigated in~\cite{scokaert2002suboptimal,graichen2010stability,rubagotti2014stabilizing}.
In~\cite{karapetyan2023finite}, the authors study the finite-time behavior of suboptimal linear MPC and provide insights into the trade-offs between computational efficiency and control performance.
This work has been extended in~\cite{karapetyan2025closed} to nonlinear systems.
In~\cite{zanelli2021lyapunov}, a suboptimal MPC scheme is proposed with theoretical guarantees on closed-loop stability obtained by deriving an upper bound on the sampling time.
More recently,~\cite{chen2025sampled} establishes the existence of a sampling-time bound under which the closed-loop system achieves exponential stability.

The main contribution of this paper lies in the design of an MPC architecture that combines suboptimal optimization and model reduction to reduce computational complexity.
Specifically, we focus on nonlinear systems characterized by the feedback interconnection between a main target dynamics, that we aim to optimally control, and a faster extra dynamics (e.g., unmodeled dynamics or a low-level compensator).
Inspired by~\cite{zanelli2021lyapunov} and~\cite{chen2025sampled}, we interpret the small parameter enabling the timescale separation as the sampling time of the original continuous-time dynamics.
Within this framework, we establish exponential stability of the equilibrium point of the closed-loop system resulting from the interconnection of the suboptimal and reduced-order MPC optimizer with the full-order plant.

The paper is organized as follows. 
In Section~\ref{sec:problem_formulation}, we present the problem setup.
In Section~\ref{sec:algo}, we describe the proposed suboptimal and reduced-order MPC algorithm and state the stability properties of the closed-loop system.
In Section~\ref{sec:theoretical_analysis}, we provide the theoretical analysis.
Finally, in Section~\ref{sec:simulations}, we validate the theoretical results through numerical simulations.

\paragraph*{Notation} 

The symbols $\R$ and $\N$ denote the set of real and natural numbers, respectively. 
The symbol $\R\ud{+}$ denotes the set of positive real numbers.
The symbol $I_n$ denotes the identity matrix of dimension $n$. 
The symbol $0_{n \times m}$ denotes the zero matrix of dimension $n \times m$. 
The symbol $\| \cdot \|$ denotes the Euclidean norm.
The symbol $\diag(d_1, \dots, d_N)$ denotes the block-diagonal matrix with blocks $d_1, \dots, d_N$ on the main diagonal.
The symbol $\col(v_1,\dots,v_N)$ denotes the vertical concatenation of the vectors $v_1,\dots,v_N$.

\section{SCENARIO DESCRIPTION}
\label{sec:problem_formulation}

In this section, we describe the class of systems we consider in this work.
In particular, we focus on discrete-time systems described by feedback-interconnected dynamics in the form
\begin{subequations}\label{eqp:plant}
    \begin{align}
        \xtp &= \slow(\xt, \xit, \ut, \pr)
        \label{eqp:slow_plant}
        \\
        \xitp &= \fast(\xit, \xt, \ut, \pr),
        \label{eq:fast_plant}
    \end{align}
\end{subequations}
where $\xt \in \domX \subseteq \R\du{n}$ is the target state, $\xit \in \domXi \subseteq \R\du{p}$ is the extra dynamics state, $\slow: \domX \times \domXi \times \domU \times \R\ud{+} \to \domX$ and $\fast: \domXi \times \domX \times \domU \times \R\ud{+} \to \domXi$ denote the target and extra dynamics, respectively, $\ut \in \domU \subseteq \R\du{m}$ is the control input, and $\pr > 0$ is a parameter that allows to arbitrarily modulate the relative speed between the target state $\xt$ and the extra dynamics state $\xit$.
We characterize this property as follows.
\begin{assumption}\label{ass:slowness}
    There exist $\bar{\pr}\ud{1}, L_{\slow} > 0$ such that, for all $\pr \in (0, \bar{\pr}\ud{1})$, it holds 
    \begin{align*}
        \norm{\slow(x, \xi, u, \pr) - \slow(x,\xi^\prime,u^\prime,\pr)} \!\leq\! \pr L_{\slow}\left(\norm{\xi - \xi^\prime} \!+\! \norm{u - u^\prime}\right),
    \end{align*}
    for all $x \in \domX$, $\xi, \xi^\prime \in \domXi$, and $u, u^\prime \in \domU$.
    Moreover, there exists $L_{\fast} > 0$ such that, for all $\pr \in (0, \bar{\pr}\ud{1})$, it holds
    \begin{align*}
        \norm{\fast(\xi, x, \uu, \pr) - \fast(\xi^\prime,x,\uu^\prime,\pr)} &\leq L_{\fast} \norm{\xi\!-\!\xi^\prime}
        + L_{\fast}\norm{u\!-\!u^\prime},
    \end{align*}
    for all $\xi, \xi^\prime \in \domXi$, $x \in \domX$, $u, u^\prime \in \domU$.
    \oprocend
\end{assumption}
The modeling framework in \eqref{eqp:plant} can arise, e.g., from a discretization of a continuous-time system in which a subsystem evolves faster than the other subsystem (thanks, e.g., to a proper design of a low-level controller). 
Further, as we formalize in the following assumption, we consider the case in which the extra dynamics~\eqref{eq:fast_plant} has equilibria parametrized in $(x,u)$ and that these equilibria are globally exponentially stable, uniformly in $(x,u)$.
\begin{assumption}\label{ass:ges_fast}
    There exists a $\lipeq$-Lipschitz continuous function $\xieq: \domX \times \domU \to \domXi$ such that
    \begin{align}\label{eq:xieq}
        \fast(\xieq(x, u), x, u, \pr) = \xieq(x, u),
    \end{align}
    for all $x \in \domX$, $u \in \domU$, and $\pr > 0$.
    Moreover, there exists a continuous function $\Uextra: \R^p \to \R$ with $a_4$-Lipschitz continuous gradient $\nabla \Uextra(\tilde{\xi})$ and $\bar{\pr}_2 > 0$, such that, for all $\pr \in (0, \bar{\pr}\ud{2})$, it holds
    \begin{subequations}\label{eq:U}
        \begin{align}
            &\!a_1\! \snorm{\xi \!-\! \xieq(x,u)}^2 \!\!\leq\! \Uextra(\xi \!-\! \xieq(x,u)\!) \!\leq\! a_2 \snorm{\xi \!-\! \xieq(x,u)}^2\!\!
            \label{eq:U_1}
            \\
            &\Uextra\big(\fast(\xi, x, u,\pr)  -  \xieq(x,u)  \big)  -  \Uextra(\xi - \xieq(x,u)) \leq  -  a_3 \snorm{\xi - \xieq(x,u)}^2, 
            \label{eq:U_2}
        \end{align}
    \end{subequations}
    for all $\xi \in  \domXi$, $x\in \domX$, $u \in \domU$ and some constants $a_1, a_2, a_3, a_4 >  0$.
    \oprocend
\end{assumption}
The practical intuition behind this assumption is that, in the ideal case in which $x$ and $u$ are fixed, the fast state $\xi$ converges rapidly to a corresponding steady-state configuration $\xieq(x, u)$.
We provide in the following a practical example satisfying Assumptions~\ref{ass:slowness} and~\ref{ass:ges_fast}.
\begin{example}\label{ex:actuated_pendulum}
    We consider a pendulum actuated by an electric motor whose continuous-time dynamics reads as
    \begin{subequations}\label{eq:pendulum_dynamics}
        \begin{align}
            \dot{\theta}(t) &= \omega(t)
            \\
            \dot{\omega}(t) &= -\frac{\beta}{J}\omega(t) - \frac{m g l}{J}\sin\big(\theta(t)\big) + \frac{1}{J}K_t I(t)
            \\
            L \dot{I}(t) &= V(t) - R I(t) - K_e \omega(t),
        \end{align}
    \end{subequations}
    where $\theta(t) \in \mathbb{S}^1$, $\omega(t) \in \R$, $I(t) \in \R$, and $V(t) \in \R$ are the pendulum angle, the angular velocity, the armature current, and the applied armature voltage at time $t \in \R$, respectively.
    As for the other elements in~\eqref{eq:pendulum_dynamics}, $m, l, J > 0$ are the pendulum mass, length and total inertia, $g, \beta > 0$ are the gravitational acceleration and viscous friction coefficients, while $R, L, K_e, K_t\! > \! 0$ are the motor resistance, armature inductance, back EMF constant, and torque constant, respectively.
    By introducing $\mathrm{x}(t) \! :=\! \col(\theta(t),\omega(t))\in\mathbb{S}^1 \! \times \! \R$, $\chi(t) \!:=\! I(t)\in\R$, $\mathrm{u}(t):=V(t)\in\R$, we rewrite~\eqref{eq:pendulum_dynamics} as
    \begin{subequations}\label{eq:pendulum_dynamics_2}
        \begin{align}
            \dot{\mathrm{x}}(t) &= 
            \begin{bmatrix}
                \mathrm{x}_2(t)
                \\
                -\tfrac{\beta}{J} \mathrm{x}_2(t) - \tfrac{m g l}{J}\sin(\mathrm{x}_1(t)) + \tfrac{K_t}{J}\chi(t)
            \end{bmatrix}
            \\
            \dot{\chi}(t) &= -\frac{1}{L}\left(R \chi(t) + K_e \mathrm{x}_2(t) - \mathrm{u}(t)\right).
        \end{align}
    \end{subequations}
    By adopting a Forward-Euler discretization with sampling time $\pr > 0$ and scaling the armature inductance as $L=\tilde{L}\pr$ with $\tilde{L}>0$, we get a discretized counterpart of~\eqref{eq:pendulum_dynamics_2} given by
    \begin{subequations}\label{eq:pendulum}
        \begin{align}
            \xtp &= \xt + \pr
            \begin{bmatrix}
                \xtt
                \\
                -\tfrac{\beta}{J} \xtt - \tfrac{m g l}{J}\sin(\xot) + \tfrac{K_t}{J}\xit
            \end{bmatrix}\label{eq:pendulum_slow}
            \\
            \xitp &= \left(1-\tfrac{R}{\tilde{L}}\right)\xit - \begin{bmatrix}0 & \tfrac{K_e}{\tilde{L}}\end{bmatrix}\xt + \tfrac{1}{\tilde{L}}\ut,
            \label{eq:pendulum_fast}
        \end{align}
    \end{subequations}
    where $\xt \in \mathbb{S}^1 \times \R$, $\xit \in \R$, and $\ut \in \R$ are the discrete-time counterparts at iteration $\iter$ of $\mathrm{x}(t)$, $\chi(t)$, and $\mathrm{u}(t)$, respectively.
    System~\eqref{eq:pendulum} is an instance of the generic system~\eqref{eqp:plant} in which the explicit expressions of $\slow$ and $\fast$ are given by
    \begin{subequations}\label{eq:slow_fast_pendulum}
        \begin{align}
            \slow(x,\xi,u,\pr) &= x + \pr
            \begin{bmatrix}
                \x_2
                \\
                -\tfrac{\beta}{J} \x_2 - \tfrac{m g l}{J}\sin(\x_1) + \tfrac{K_t}{J}\xi
            \end{bmatrix}
            \label{eq:slow_definition_pendulum}
            \\
            \fast(\xi,x,u) &= \left(1-\tfrac{R}{\tilde{L}}\right)\xi - \begin{bmatrix}0 & \tfrac{K_e}{\tilde{L}}\end{bmatrix}x + \tfrac{1}{\tilde{L}}u.
            \label{eq:fast_definition_pendulum}
        \end{align}
    \end{subequations}
    Indeed, by observing~\eqref{eq:slow_fast_pendulum}, we can see that system~\eqref{eq:pendulum} satisfies Assumption~\ref{ass:slowness} with $L_{\slow} = K_t/J$ and $L_{\fast} = \max\{|1-R/\tilde{L}|,K_e/\tilde{L},1/\tilde{L}\}$.
    Moreover, by observing~\eqref{eq:fast_definition_pendulum}, we also note that subsystem~\eqref{eq:pendulum_fast} admits an equilibrium function $\xieq: \mathbb{S}^1 \times \R \times \R \to \R$ defined as
    \begin{align}
        \xieq(x,u) = \frac{1}{R}\left(-\begin{bmatrix}0 & K_e\end{bmatrix}x + u\right),\label{eq:xieq_specific}
    \end{align}
    which satisfies Assumption~\ref{ass:ges_fast} with $\lipeq = \max\{1,K_e\}/R$.
    Further, for all $\tilde{L} > R/2$, the equilibrium $\xieq(x,u)$ is globally exponentially stable uniformly in $(x,u) \in \mathbb{S}^1 \times \R \times \R$ and, thus, Assumption~\ref{ass:ges_fast} is satisfied.
\end{example}

\begin{remark}
    We emphasize the dual role of the parameter $\pr$ in our framework. 
    In classical continuous-time singular perturbation theory (see, e.g.,~\cite{abdelgalil2023multi}), timescale separation is typically induced by a small physical parameter that accelerates the fast subsystem. 
    In the proposed discrete-time setting, this role is naturally played by the sampling time $\pr$. 
    Indeed, as illustrated by the discretization process in Example~\ref{ex:actuated_pendulum}, choosing a sufficiently small sampling time induces, from a mathematical viewpoint, a timescale separation between the fast extra dynamics and the slow target one, provided that the speed of the former can be arbitrarily tuned with respect to $\pr$. \oprocend
\end{remark}

\section{\algoCAPS/: STRATEGY DESCRIPTION AND \\STABILITY GUARANTEES}
\label{sec:algo}

In this section, we introduce the proposed \algo/ strategy and provide the corresponding stability guarantees for the resulting closed-loop system.

\subsection{MPC Strategy Description}

By exploiting the system structure outlined in Section~\ref{sec:problem_formulation}, we design an MPC strategy that operates on a reduced-order model of system~\eqref{eqp:plant}.
Specifically, we approximate the full dynamics by a reduced dynamics map $\red: \domX \times \domU \times \R_{+} \to \domX$
that captures the target subsystem $\slow$ evaluated on the extra dynamics manifold $\xi = \xieq(x, u)$, namely
\begin{align}\label{eq:red}
    \red(x, u, \pr) := \slow(x, \xieq(x, u), u, \pr).
\end{align}
The control input $\ut$ is iteratively computed through a suboptimal MPC scheme applied to the reduced system, thereby adhering to and further reinforcing the rationale of computational efficiency motivating the use of $\red$.
Specifically, given a time horizon $\hor \in \N$, at each iteration $\iter \in \N$, we consider the finite-horizon optimal control problem
\begin{align}\label{eq:MPC_problem}
    \begin{aligned}
        \min_{\{\uta\}_{\initer=0}^{\hor-1}} \quad & \sum_{\initer=0}^{\hor-1} \cost\left(\xta, \uta\right) + \costf\left(\x\ud{\hor}\right) 
        \\
        \text{s.t.} \quad
            & x\ud{0} = \xt
            \\%
            & \xtap = \red(\xta, \uta, \prm), \quad \initer = 0, \ldots, \hor-1
            \\
            & \uta \in \mathcal{U}, \quad \xtap \in \mathcal{X}, 
            \quad \initer = 0, \ldots, \hor-1 
            \\
            & x\ud{\hor} \in \mathcal{X}_f, %
    \end{aligned}
\end{align}
where $\cost: \domX \times \domU \to \R\ud{+}$ is the stage cost, $\costf: \domX \to \R\ud{+}$ is the terminal cost, $\mathcal{U} \subseteq \domU$ is the input constraint set, $\mathcal{X} \subseteq \domX$ is the state constraint set, $\mathcal{X}_f \subseteq \domX$ is the terminal constraint set, and $\prm > 0$ is the sampling time used by the optimizer.
Problem \eqref{eq:MPC_problem} is addressed suboptimally.
More in detail, we consider an optimizer dynamics of the form
\begin{subequations}\label{eq:mpc_dyn}
    \begin{align}
        \ztp &= \mpc(\zt, \xt) 
        \\
        \ut &= \Pi(\zt),
    \end{align}
\end{subequations}
where $\zt \in \domZ \subseteq \R^{Tm}$ is the optimizer state, $\mpc: \domZ \times \domX \to \domZ$ describes the optimization algorithm dynamics, and $\hor \in \N$ is the prediction horizon.
As customary in MPC schemes, only the first element of the $\zt$ sequence is applied to the plant as control input, hence the projection map $\Pi: \domZ \to \domU$ in \eqref{eq:mpc_dyn} reads as 
\begin{align}\label{eq:projection_map}
    \Pi(z) := \left[
        I_m \quad 0_{m \times (T-1)m}
    \right]z.
\end{align}
In the following, some assumptions on the optimization algorithm are provided. 
By denoting with $\zstar(\x)$ the unique solution to problem \eqref{eq:MPC_problem} at given initial conditions $\x$, we assume as follows that $\zstar(\x)$ is a globally exponentially stable equilibrium point of the optimizer dynamics~\eqref{eq:mpc_dyn}.

\begin{assumption}\label{ass:ges_MPC}
    For all $x \in \domX$, it holds 
    \begin{align}\label{eq:optimum_equilibrium}
        \zstar(x) = \mpc(\zstar(x), x).
    \end{align}
    Moreover, there exist a continuously differentiable function $\cL: \R^{\hor m} \to \R_+$ with $b_4$-Lipschitz continuous gradient $\nabla_z \cL(z)$ and constants $b_1, b_2, b_3, b_4>0$ such that
        \begin{subequations}\label{eq:cL}
                \begin{align}
                &b_1\norm{z \!\!-\! \zstar(x)}^2 \!\leq\! \cL(z \!\!-\! \zstar(x)) \leq b_2\norm{z \!\!-\! \zstar(x)}^2
                \label{eq:cL_1}
                \\
                &\cL(\mpc(z,x) \!\!-\! \zstar(x)) \!-\! \cL(z \!\!-\! \zstar(x)) \!\leq\! -b_3\norm{z \!\!-\! \zstar(x)}^2\!,\!
                \label{eq:cL_2}
            \end{align}
        \end{subequations}
        for all $\z \in \domZ$ and $\x \in \domX$.
    \oprocend
\end{assumption}
Furthermore, we assume that the optimal solution of~\eqref{eq:MPC_problem} yields an exponentially stable equilibrium $\xstar \in \domX$ for the reduced system.

\begin{assumption}
    \label{ass:reduced}
    There exist a continuously differentiable function $W: \domX \to \R$ with $c_4$-Lipschitz continuous gradient $\nabla_x W(x)$, constants $c_1, c_2, c_3 > 0$, and $\bar{\pr}_3 > 0$ such that, for all $\pr \in (0, \bar{\pr}_3)$, it holds
    \begin{subequations}\label{eq:W}
         \begin{align}
            &c_1 \norm{x - \xstar}^2 \leq W(x) \leq c_2 \norm{x - \xstar}^2
            \\
            &W(\red(x,\Pi(\zstar(x)),\pr)) \!-\! W(x) \!\leq\! - \pr c_3 \norm{x - \xstar}^2\!,\!
        \end{align}
    \end{subequations}
    for all $x \in \domX$.
    Moreover, for all $\pr \in (0,\bar{\pr}_3)$ and some $L_{\red} > 0$, it holds 
    \begin{align}
        \norm{\red(x,\Pi(\zstar(x)),\pr) - x} &\leq \pr L_{\red}\norm{x - \xstar},
    \end{align}
    for all $x \in \domX$.
    \oprocend
\end{assumption}
\begin{remark}
    Sufficient conditions on the underlying optimal control problem~\eqref{eq:MPC_problem} can be established to guarantee Assumption~\ref{ass:reduced}, namely, exponential stability of $\xstar$ for the reduced-order dynamics $\xtp = \red(\xt,\ut, \pr)$ under the optimal control law $\ut = \Pi(\zstar(\xt))$. 
    These conditions regard, e.g., appropriate choice of stage and terminal cost functions, terminal constraints, sufficiently long prediction horizon, and regularity conditions on the system dynamics, see, e.g.,~\cite[Chap.~2, pp.~139--141]{rawlings2020model}.
    \oprocend
\end{remark}

\subsection{Closed-Loop System: Stability Guarantees}
\label{sec:algorithm_presentation}

The closed-loop system is thus described by the interconnection of the plant~\eqref{eqp:plant} and the optimizer~\eqref{eq:mpc_dyn}, namely
\begin{subequations}\label{eq:interconnected_sys}
    \begin{align}
        \xtp &= \slow(\xt, \xit, \Pi(\zt), \pr)
        \label{eq:interconnected_system_x}
        \\
        \xitp &= \fast(\xit, \xt, \Pi(\zt),\pr)
        \label{eq:interconnected_system_xi}
        \\
        \ztp &= \mpc(\zt, \xt),
        \label{eq:interconnected_system_z}
    \end{align}
\end{subequations}
Fig.~\ref{fig:bd_pendulum} shows a schematic representation of the closed-loop system in the pendulum-motor setup in Example~\ref{ex:actuated_pendulum}.
\begin{figure}[H]
    \centering
    \includegraphics[scale=.74]{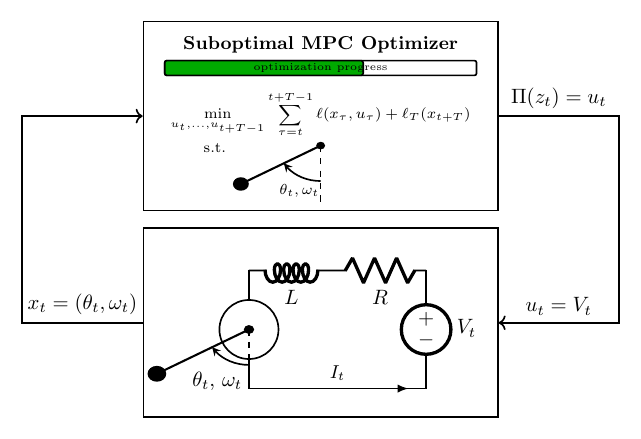}
    \caption{Block diagram representation of the closed-loop system resulting from the interconnection of the full plant dynamics~\eqref{eqp:plant} and the reduced-order suboptimal MPC~\eqref{eq:mpc_dyn} for the pendulum-motor setup (cf.~Example~\ref{ex:actuated_pendulum}).
    }
    \label{fig:bd_pendulum}
\end{figure}
The stability properties of the closed-loop system~\eqref{eq:interconnected_sys} are formalized as follows.
\begin{theorem}\label{th:main}
    Consider the closed-loop system~\eqref{eq:interconnected_sys} and let Assumptions~\ref{ass:slowness},~\ref{ass:ges_fast},~\ref{ass:ges_MPC}, and~\ref{ass:reduced} hold.
    Then, there exists $\bar{\pr} \in (0, \min\{\bar{\pr}_1, \bar{\pr}_2, \bar{\pr}_3 \})$ such that, for all $\pr \in (0, \bar{\pr})$, the point $(\xstar,\xieq(\xstar,\Pi(\zstar(\xstar))),\zstar(\xstar))$ is a globally exponentially stable equilibrium of system~\eqref{eq:interconnected_sys}. \oprocend
\end{theorem}
The proof of Theorem~\ref{th:main} will be provided in a forthcoming document. A sketch of the proof is provided in Section \ref{sec:proof}.

\section{THEORETICAL ANALYSIS}
\label{sec:theoretical_analysis}

In this section, we analyze the closed-loop system arising from the interconnection of the plant dynamics~\eqref{eqp:plant} and the optimizer one~\eqref{eq:mpc_dyn} to prove Theorem~\ref{th:main}.
Specifically, we interpret this closed-loop system as a two-time-scale system in which the fast subsystem includes both the plant extra dynamics~\eqref{eq:fast_plant} and the optimizer dynamics~\eqref{eq:mpc_dyn}, while the slow part includes only the plant target dynamics~\eqref{eqp:slow_plant}.
This feedback interconnection is graphically depicted in Fig.~\ref{fig:bd_interconnection}. 
\begin{figure}[H]
  \centering
  \includegraphics[scale=\scale]{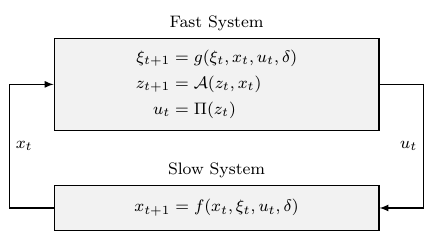}
  \caption{Block diagram representation of the interconnected system.}
  \label{fig:bd_interconnection}
\end{figure}
Hence, as customary in two-time-scale systems analysis, we study (i) the fast subsystem in the ideal case in which the slow state $\xt$ is fixed  (cf. Section~\ref{sec:bl}) and, conversely, (ii) the slow one in the ideal case in which the fast state $(\xit, \zt)$ lies in its equilibrium manifold (cf. Section~\ref{sec:reduced}).
Finally, in Section~\ref{sec:proof} we combine the results obtained in this preparatory phase to prove Theorem~\ref{th:main}.

\subsection{Boundary-Layer System Analysis}
\label{sec:bl}

Here, we focus on the boundary-layer system, i.e., the subsystem obtained by considering an arbitrarily fixed slow state $\xt = x$ for all $\iter \in \N$ in the fast subsystem~\eqref{eq:interconnected_system_xi}-\eqref{eq:interconnected_system_z}.
A graphical representation of this system is shown in Fig.~\ref{fig:bd_boundary_layer}.
\begin{figure}[H]
  \centering
  \includegraphics[scale=\scale]{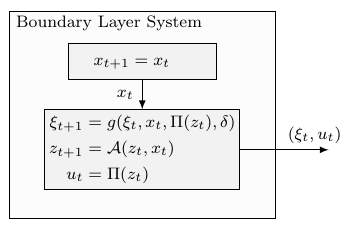}
  \caption{Block diagram representation of the boundary layer system.}
  \label{fig:bd_boundary_layer}
\end{figure}
Hence, by observing~\eqref{eq:interconnected_system_xi}-\eqref{eq:interconnected_system_z} and introducing the error coordinates $\col(\txit,\tzt) := \col(\xit - \xieq(x, \Pi(\zt)), \zt - \zstar(x))$, the boundary-layer system reads as
\begin{subequations}\label{eq:BL_not_compact}
    \begin{align}
        \txitp &= \fast(\txit + \xieq(x, \Pi(\tzt + \zstar(x))), x, \Pi(\tzt + \zstar(x)),\pr)  - \xieq(x, \Pi(\mpc(\tzt + \zstar(x), x)))
        \\
        \tztp &= \mpc(\tzt + \zstar(x), x) - \zstar(x).
    \end{align}
\end{subequations}
For the sake of compactness, we define the functions
\begin{subequations}\label{eq:tilde_definitions}
    \begin{align}
        \tfast(\txi, x, \tz, \pr)
        &:= \fast(\txi + \xieq(x, \Pi(\tz + \zstar(x))), x, \Pi(\tz + \zstar(x)),\pr)
         - \xieq(x, \Pi(\tz + \zstar(x)))
        \label{eq:tfast}
        \\
        \tmpc(\tz, x) &:= \mpc(\tz + \zstar(x), x) - \zstar(x)
        \\
        \Delta \xieq(x, \tz,\tz^\prime) &:= \xieq(x, \Pi(\tz + \zstar(x))) 
        - \xieq(x, \Pi(\tz^\prime + \zstar(x))),\label{eq:delta_xieq}
    \end{align}
\end{subequations}
which allow us to compactly rewrite~\eqref{eq:BL_not_compact} as
\begin{subequations}\label{eq:BL}
    \begin{align}
        \txitp &= \tfast(\txit, x, \tzt,\pr) \! + \! \Delta\xieq(x, \tztp,\tzt)
        \\
        \tztp &= \tmpc(\tzt, x).
    \end{align}
\end{subequations}
The following lemma ensures global exponential stability of the origin for the boundary-layer system~\eqref{eq:BL}.
To properly state this result, let us introduce the set $\cS$ defined as
\begin{align*} 
    \cS := &\{(\txi, x, \tz) \in \R^{p} \times \domX \times \R^{\hor m} \mid 
   \tilde{\xi} + \xieq(x, \Pi(\tilde{z}  + \zstar(x))) \in \domXi, \tilde{\z} + \zstar(x) \in \domZ\}.
\end{align*}

\begin{lemma}\label{lemma:bl}
    There exists a continuous function $U: \R^p \times \R^{\hor m} \to \R$ such that, for all $\pr \in (0, \min \{\bar{\pr}_1, \bar{\pr}_2 \})$, it holds
    \begin{subequations}\label{eq:cU}
        \begin{align}
        d_1 (\snorm{\txi}^2 + \snorm{\tz}^2) &\leq U(\txi, \tz) \leq d_2 (\snorm{\txi}^2 + \snorm{\tz}^2)
        \label{eq:cU_1}
        \\
        U\big(\txi\ud{+},\tz\ud{+}\big) - U(\txi,\tz) &\leq -d_3 (\snorm{\txi}^2 + \snorm{\tz}^2)\label{eq:cU_2}
        \\
        |U(\txi, \tz) - U(\txi^\prime, \tz^\prime)| &\leq d_4 \norm{\col(\txi, \tz) - \col(\txi^\prime, \tz^\prime)} 
        \label{eq:cU_3}(\snorm{\col(\txi,\tz)} + \snorm{\col(\txi^\prime,\tz^\prime)}),
        \end{align}
    \end{subequations}
    for all $(\txi, x, \tz) \in \cS$, and some constants $d_1, d_2, d_3, d_4 > 0$, where $\txi\ud{+} := \tfast(\txi, x,\tz,\pr) + \Delta\xieq(x, \tz,\tmpc(\tz, x))$ and $\tz\ud{+} := \tmpc(\tz, x)$.
    \oprocend
\end{lemma}

The proof of Lemma~\ref{lemma:bl} will be provided in a forthcoming document. %

\subsection{Reduced System Analysis}
\label{sec:reduced}

As graphically depicted in Fig.~\ref{fig:bd_reduced_system}, the reduced system corresponds to the slow subsystem~\eqref{eq:interconnected_system_x} studied by considering the fast state $(\xit, \zt)$ in its equilibrium manifold, i.e., $\xit = \xieq(\xt, \Pi(\zt))$ and $\zt = \zstar(\xt)$ for all $\iter \in \N$.
\begin{figure}[H]
  \centering
  \includegraphics[scale=\scale]{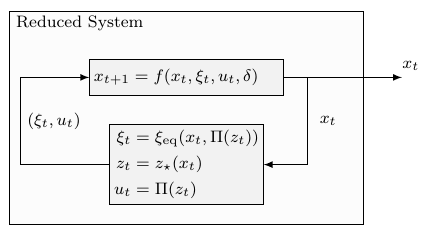}
  \caption{Block diagram representation of the reduced system.}
  \label{fig:bd_reduced_system}
\end{figure}
Then, by definition of $\red$ in \eqref{eq:red}, the reduced system reads
\begin{align}\label{eq:reduced_system}
    \xtp = \red(\xt, \Pi(\zstar(\xt)), \pr).
\end{align}
Therefore, the reduced system convergence properties directly follow from Assumption~\ref{ass:reduced}.

\subsection{Sketch of proof of Theorem~\ref{th:main}}
\label{sec:proof}

    The proof relies on timescale-separation arguments applied to the closed-loop system~\eqref{eq:interconnected_sys} to establish the stability of its equilibrium.
    More precisely, the analysis combines: (i) the stability properties of the boundary-layer system (cf. Lemma~\ref{lemma:bl}), (ii) those of the reduced system (cf. Assumption~\ref{ass:reduced}), and (iii) a careful treatment of the actual interconnected dynamics, rather than the ``ideal'' dynamics considered in these two auxiliary systems.
    Indeed, we recall that the boundary-layer system neglects the drift of its equilibrium, while, analogously, the reduced system assumes that the fast state lies at all times on the equilibrium manifold.
    Therefore, these two ingredients must be combined appropriately, together with a suitable tuning of $\pr$ and appropriate regularity assumptions on the dynamics and on the equilibrium map, in order to control the variation of $\xt$ and derive an upper bound $\bar{\pr}$ below which stability of the interconnected system is preserved.

\section{NUMERICAL SIMULATIONS}
\label{sec:simulations}

In this section, we present numerical simulations to validate the proposed \algo/ control strategy.
In particular, we consider the actuated pendulum introduced in Example~\ref{ex:actuated_pendulum}, whose physical parameter values are reported in Table~\ref{tab:physical_parameters}. 
All simulations were conducted using the CasADi tool for numerical optimization \cite{Andersson2019}.
\begin{table}[htpb]
    \renewcommand{\arraystretch}{1.1}
\caption{Physical Parameters of the Actuated Pendulum}
\label{tab:physical_parameters}
\begin{center}
\begin{tabular}{|c|c|}
\hline
\textbf{Parameter} & \textbf{Value} \\
\hline
$l$ & $1.0 [\text{m}]$\\
\hline
$m$ & $0.5 [\text{Kg}]$ \\
\hline
$\beta$ & $0.5 [\text{N m s/rad}]$  \\
\hline
$J$ & $0.5 [\text{Kg m}^2]$  \\
\hline
$K_t$ & $0.4 [\text{N m/A}]$ \\
\hline
$K_e$ & $0.4 [\text{V s/rad}]$ \\
\hline
$R$ & $0.6 [\Omega]$\\
\hline
$L$ & $\tilde{L} \pr~[\text{H}]$\\
\hline
$\tilde{L}$ & $1.0~[\text{H}]$\\
\hline
\end{tabular}
\end{center}
\end{table}
Consistently with the theoretical setting, the optimizer is only aware of the reduced-order model obtained by considering~\eqref{eq:pendulum_slow} in the case in which the armature current $\xit$ lies in its equilibrium manifold~\eqref{eq:xieq_specific}.
Namely, in this scenario, the reduced dynamics map~\eqref{eq:red} explicitly reads as
\begin{align*}
    \red(\x, \uu, \pr) \!:=\! 
    \begin{bmatrix}
        x_1 + \pr x_2 \\
        x_2 \!+\!\! \pr \left(
            -\frac{g}{l} \sin(x_1) \!-\! \frac{\beta R + K_t K_e}{m l^2 R}x_2 + \frac{K_t}{m l^2 R} \uu
        \right)
    \end{bmatrix}\!.
\end{align*}
In all simulations, we consider $\hor=4$ and cost functions and saturation constraints given by
\begin{align*}
    \cost(x,u) := x^\top Q x + 0.01 u^2, \quad
     \costf(x) := x^{\top} Q_f x
\end{align*}
with $Q = Q_f := \diag(100, 0.1)$, $R := 0.01$. The system is subject to the actuation constraints: $\uu \in \mathcal{U} := \{\uu \in \R \mid |\uu| \leq 12\}$.
The MPC problem is addressed suboptimally through a single step of a projected gradient method.
We compare the proposed approach, where the suboptimal MPC optimizer relies on a reduced-order prediction model while interacting with the actual full-order plant, against two benchmark settings. In these benchmarks, the MPC operates either optimally or suboptimally, but the plant is artificially forced to operate under an ideal steady-state regime for its fast dynamics, i.e., $\xit = \xieq(\xt, \ut)$ for all $t \in \mathbb{N}$. Because this idealized plant behavior perfectly aligns with the reduced-order prediction model, no structural plant-model mismatch exists for the optimizer in these benchmarks. To accurately reflect this condition, we refer to these benchmarks as \textit{Optimal (Ideal Plant)} and \textit{Suboptimal (Ideal Plant)}.
These strategies are compared in three scenarios in which different $\pr$ are used.
As predicted by Theorem~\ref{th:main}, Fig.~\ref{fig:PG_all_work} shows that the closed-loop system exhibits exponential convergence towards the desired equilibrium configuration when a sufficiently small $\pr = 0.01s$ is used. %
Fig.~\ref{fig:PG_only_SS_works} highlights the importance of timescale separation. 
Indeed, for a larger value of $\pr$ ($\pr = 0.12s$), the suboptimal MPC still provides acceptable closed-loop behavior when controlling the idealized matched plant but its performance degrades significantly when interacting with the actual full-order plant.
Finally, Fig.~\ref{fig:PG_only_SSopt_works} further confirms the need for timescale separation.
In this case, in which $\pr = 0.2s$, only the exact optimal MPC applied to the idealized matched plant guarantees good closed-loop performance. Conversely, the suboptimal numerical optimization leads to a severe performance degradation even in the absence of plant-model mismatch, a deterioration that is further exacerbated when controlling the full-order plant.
\begin{figure}[H]
    \centering
    \includegraphics[width=0.98\columnwidth]{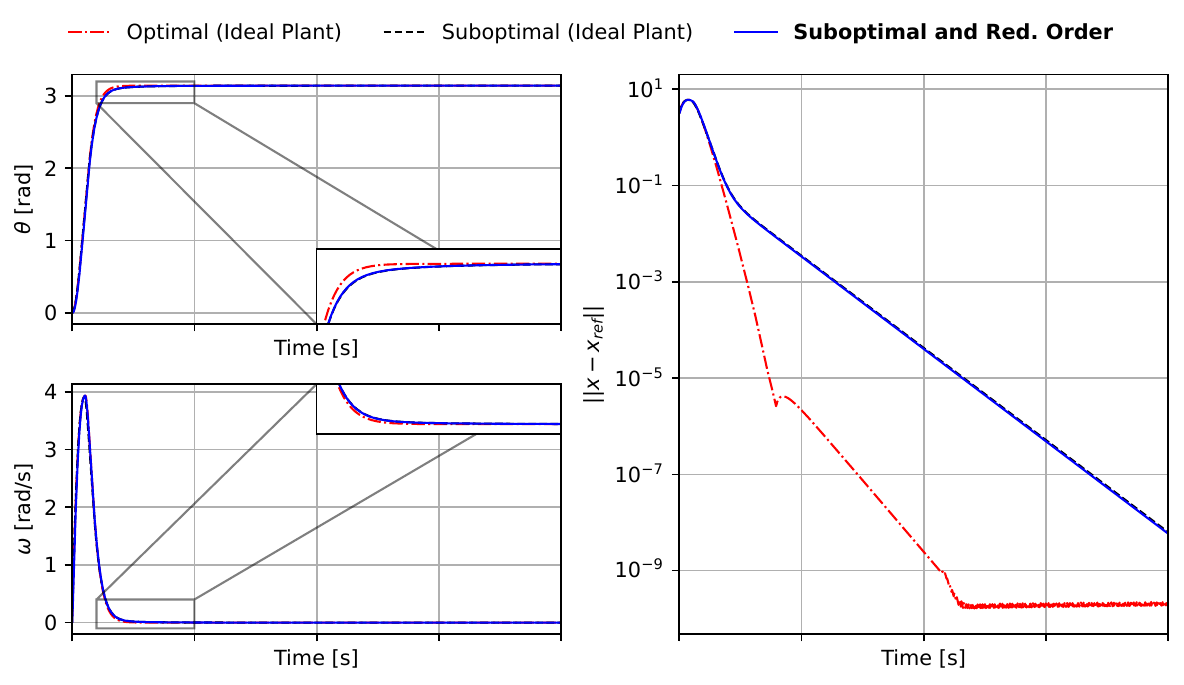}
    \caption{Closed-loop trajectories of the target states $\theta$ and $\omega$ (left) and state error in semilog scale (right) for $\pr = 0.01$ s.}
    \label{fig:PG_all_work}
\end{figure}

\begin{figure}[H]
    \centering
    \includegraphics[width=0.98\columnwidth]{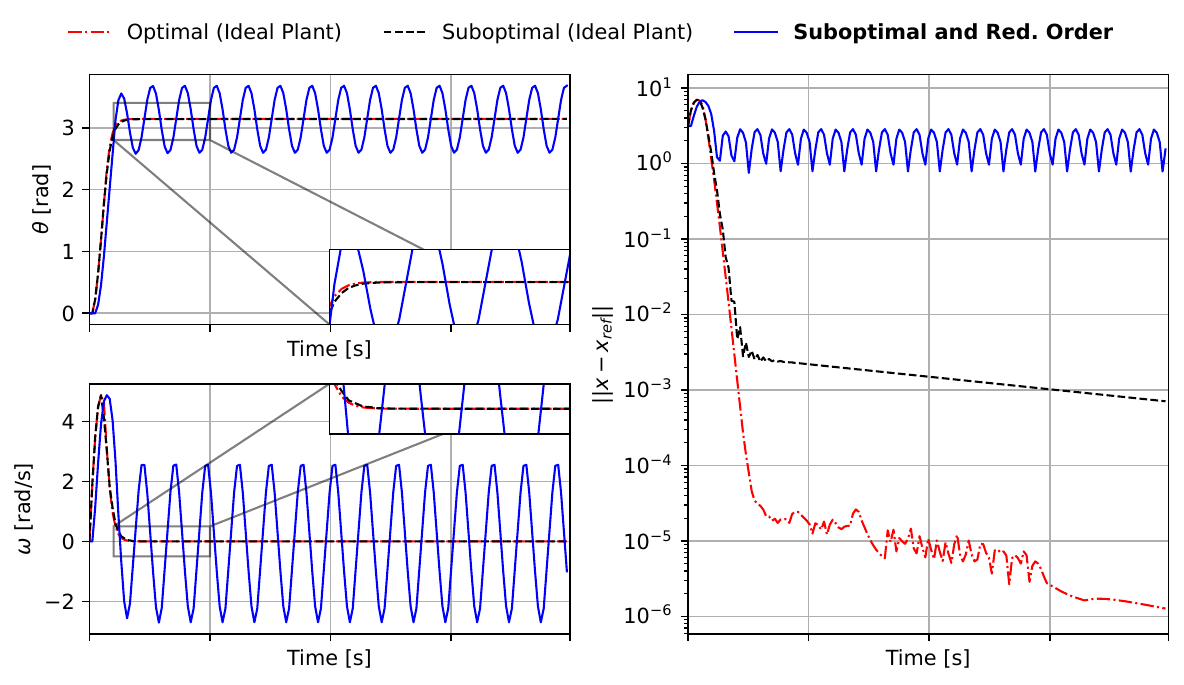}
    \caption{Closed-loop trajectories of the target states $\theta$ and $\omega$ (left) and state error in semilog scale (right) for $\pr = 0.12$ s.}
    \label{fig:PG_only_SS_works}
\end{figure}
\begin{figure}[H]
    \centering
    \includegraphics[width=0.98\columnwidth]{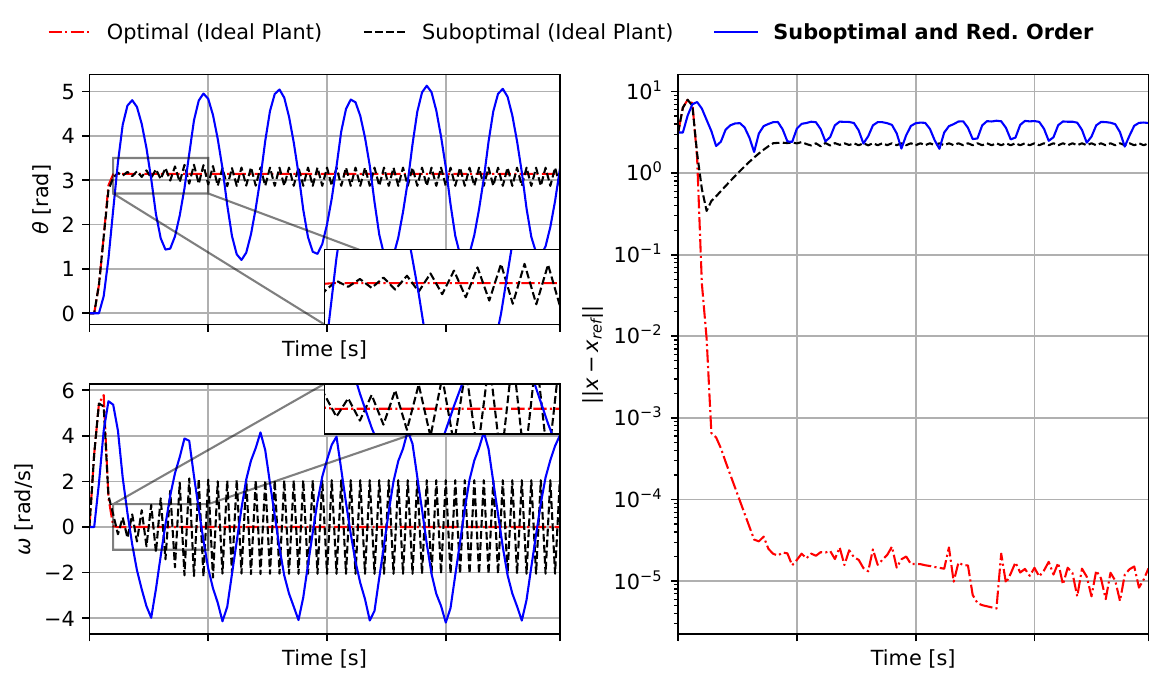}
    \caption{Closed-loop trajectories of the target states $\theta$ and $\omega$ (left) and state error in semilog scale (right) for $\pr = 0.2$ s.}
    \label{fig:PG_only_SSopt_works}
\end{figure}

\section{CONCLUSIONS}
\label{sec:conclusions}

We proposed a suboptimal and reduced-order MPC scheme for discrete-time interconnected systems. 
We relied on a reduced-order model that omits part of the plant dynamics, enabling a lighter computational burden while maintaining closed-loop stability guarantees. 
In detail, we showed that stability is preserved by properly tuning a timescale parameter. 
Numerical tests on an actuated pendulum confirmed the effectiveness of our approach.
A rigorous treatment of recursive feasibility is left for future work.

%
                                  %
                                  %
                                  %
                                  %
                                  %

%

%

%
%

%
%

%
%
%
%
%
%
%
%
%
%
%
%
%
%
%
%
%
%
%
%
%
%
%
%
%
%
%
%
%
%
%
%
%
%
%
%
%
%
%
%
%
%
%
%
%
%
%
%
%
%
%
%
%
%
%
%
%
%
%
%
%
%
%
%
%
%
%
	
%
%
%
%
%
%
%
%
%
%
%
%
%
%
%
%
%
%
%
%
%
%
%
%
%
%
%
%
%
%
%
%
%
%
%
%
%
%
%
%
%
%
%
%
%
%
%
%
%
%
%
%
%

%
%
%
%
%
%
%
%
%
%
%
%
%
%
%
%
%
%
%
%
%
%
%
%
%
%
%
%
%
%
%
%
%
%
%
%
%
%
%
%
%
%
%
%
%
%
%
%
%
%
%
%
%
%
%
%
%
%
%
%
%
%
%
%
%
%
%
%
%
%
%
%
%
%
%
%
%
%
%
%
%
%
%
%
%
%
%
%
%
%
%
%
%
%
%
%
%
%
%
%
%
%
%
%
%
%
%
%
%
%
%
%
%
%
%
%
%
%
%
%
%
%
%
%
%
%
%
%
%
%
%
%
%
%
%
%
%
%
%
%
%
%
%
%
%
%
%
%
%
%
%
%
%
%
%
%
%
%
%
%
%
%
%
%
%
%
%
%
%
%
%
%
%
%
%
%
%
%
%
%
%
%
%
%
%
%
%
%
%
%
%
%
%
%
%
%
%
%
%
%
%
%
%
%
%
%
%
%
%
%
%
%
%
%
%
%
%
%
%
%
%
%
%
%
%
%
%
%
%
%
%
%
%
%
%
%
%
%
%
%
%
%
%
%
%
%
%
%
%
%
%
%
%
%
%
%
%
%
%
%
%
%
%
%
%
%
%
%
%
%
%
%
%
%
%
%
%
%
%
%
%
%
%
%
%
%
%
%
%
%
%
%
%
%
%
%
%
%
%
%
%
%
%
%
%
%
%
%
%
%
%
%
%
%
%
%
%
%
%
%
%
%
%
%
%
%
%
%
%
%
%
%
%
%
%
%
%
%
%
%

%
%

%
%
%
%
%
%
%
%
%
%
%
%
%
%
%
%
%
%
%
%
%
%
%
%
%
%
%
%
%
%
%
%
%
%
%
%
%
%
%
%
%
%
%
%
%
%
%
%
%
%
%
%
%
%
%
%
%
%
%
%
%
%
%
%
%
%
%
%
%
%
%
%
%
%
%
%
%
%
%
%
%
%
%
%
%
%
%
%
%
%
%
%
%
%
%
%
%
%
%
%
%
%
%
%
%
%
%
%
%
%
%
%
%
%
%
%
%
%
%
%
%
%
%
%
%
%
%
%
%
%
%
%
%
%
%
%
%
%
%
%
%
%
%
%
%
%
%
%
%
%
%
%
%
%
%
%
%
%
%
%
%
%
%
%
%
%
%
%
%
%
%
%
%
%
%
%
%
%
%
%
%
%
%
%
%
%
%
%
%
%
%
%
%
%
%
%
%
%
%
%
%
%
%
%
%
%
%
%
%
%
%
%
%
%
%
%
%
%
%
%
%
%
%
%
%
%
%
%
%
%
%
%
%
%
%
%
%
%
%
%
%
%
%
%
%
%
%
%
%
%
%
%
%
%
%
%
%
%
%
%
%
%
%
%
%
%
%
%
%
%
%
%
%
%
%
%
%
%
%
%
%
%
%
%
%
%
%
%
%
%
%
%
%
%
%
%
%
%
%
%
%
%
%
%
%
%
%
%
%
%
%
%
%
%
%
%
%
%
%
%
%
%
%
%
%
%
%
%
%
%
%
%
%

%
%
%
%
%
%

%
%
%

%

%


\begin{thebibliography}{10}
\providecommand{\url}[1]{#1}
\csname url@samestyle\endcsname
\providecommand{\newblock}{\relax}
\providecommand{\bibinfo}[2]{#2}
\providecommand{\BIBentrySTDinterwordspacing}{\spaceskip=0pt\relax}
\providecommand{\BIBentryALTinterwordstretchfactor}{4}
\providecommand{\BIBentryALTinterwordspacing}{\spaceskip=\fontdimen2\font plus
\BIBentryALTinterwordstretchfactor\fontdimen3\font minus \fontdimen4\font\relax}
\providecommand{\BIBforeignlanguage}[2]{{%
\expandafter\ifx\csname l@#1\endcsname\relax
\typeout{** WARNING: IEEEtran.bst: No hyphenation pattern has been}%
\typeout{** loaded for the language `#1'. Using the pattern for}%
\typeout{** the default language instead.}%
\else
\language=\csname l@#1\endcsname
\fi
#2}}
\providecommand{\BIBdecl}{\relax}
\BIBdecl

\bibitem{bemporad2006model}
A.~Bemporad, ``Model predictive control design: New trends and tools,'' in \emph{Proceedings of the 45th IEEE Conference on Decision and Control}.\hskip 1em plus 0.5em minus 0.4em\relax IEEE, 2006, pp. 6678--6683.

\bibitem{kouvaritakis2016model}
B.~Kouvaritakis and M.~Cannon, ``Model predictive control,'' \emph{Switzerland: Springer International Publishing}, vol.~38, no. 13-56, p.~7, 2016.

\bibitem{rawlings2020model}
J.~B. Rawlings, D.~Q. Mayne, M.~Diehl \emph{et~al.}, \emph{Model predictive control: theory, computation, and design}.\hskip 1em plus 0.5em minus 0.4em\relax Nob Hill Publishing Madison, WI, 2020, vol.~2.

\bibitem{rosolia2020multi}
U.~Rosolia and A.~D. Ames, ``Multi-rate control design leveraging control barrier functions and model predictive control policies,'' \emph{IEEE Control Systems Letters}, vol.~5, no.~3, pp. 1007--1012, 2020.

\bibitem{csomay2022multi}
N.~Csomay-Shanklin, A.~J. Taylor, U.~Rosolia, and A.~D. Ames, ``Multi-rate planning and control of uncertain nonlinear systems: Model predictive control and control {Lyapunov} functions,'' in \emph{2022 IEEE 61st Conference on Decision and Control (CDC)}.\hskip 1em plus 0.5em minus 0.4em\relax IEEE, 2022, pp. 3732--3739.

\bibitem{rosolia2022unified}
U.~Rosolia, A.~Singletary, and A.~D. Ames, ``Unified multirate control: From low-level actuation to high-level planning,'' \emph{IEEE Transactions on Automatic Control}, vol.~67, no.~12, pp. 6627--6640, 2022.

\bibitem{matni2024towards}
N.~Matni, A.~D. Ames, and J.~C. Doyle, ``Towards a theory of control architecture: A quantitative framework for layered multi-rate control,'' \emph{arXiv preprint arXiv:2401.15185}, 2024.

\bibitem{loehning2014model}
M.~Loehning, M.~Reble, J.~Hasenauer, S.~Yu, and F.~Allgoewer, ``Model predictive control using reduced order models: Guaranteed stability for constrained linear systems,'' \emph{Journal of Process Control}, vol.~24, no.~11, pp. 1647--1659, 2014.

\bibitem{lorenzetti2019reduced}
J.~Lorenzetti, B.~Landry, S.~Singh, and M.~Pavone, ``Reduced order model predictive control for setpoint tracking,'' in \emph{2019 18th European Control Conference (ECC)}.\hskip 1em plus 0.5em minus 0.4em\relax IEEE, 2019, pp. 299--306.

\bibitem{lorenzetti2020error}
J.~Lorenzetti and M.~Pavone, ``Error bounds for reduced order model predictive control,'' in \emph{2020 59th IEEE Conference on Decision and Control (CDC)}.\hskip 1em plus 0.5em minus 0.4em\relax IEEE, 2020, pp. 2521--2528.

\bibitem{kartmann2024certified}
M.~Kartmann, M.~Manucci, B.~Unger, and S.~Volkwein, ``Certified model predictive control for switched evolution equations using model order reduction,'' \emph{arXiv preprint arXiv:2412.12930}, 2024.

\bibitem{wang2022tube}
W.~Wang and J.~P. Koeln, ``Tube-based robust {MPC} for two-timescale systems using reduced-order models,'' \emph{IEEE Control Systems Letters}, vol.~7, pp. 799--804, 2022.

\bibitem{diehl2005real}
M.~Diehl, H.~G. Bock, and J.~P. Schl{\"o}der, ``A real-time iteration scheme for nonlinear optimization in optimal feedback control,'' \emph{SIAM Journal on control and optimization}, vol.~43, no.~5, pp. 1714--1736, 2005.

\bibitem{gros2020linear}
S.~Gros, M.~Zanon, R.~Quirynen, A.~Bemporad, and M.~Diehl, ``From linear to nonlinear {MPC}: bridging the gap via the real-time iteration,'' \emph{International Journal of Control}, vol.~93, no.~1, pp. 62--80, 2020.

\bibitem{scokaert2002suboptimal}
P.~O. Scokaert, D.~Q. Mayne, and J.~B. Rawlings, ``Suboptimal model predictive control (feasibility implies stability),'' \emph{IEEE Transactions on Automatic Control}, vol.~44, no.~3, pp. 648--654, 2002.

\bibitem{graichen2010stability}
K.~Graichen and A.~Kugi, ``Stability and incremental improvement of suboptimal {MPC} without terminal constraints,'' \emph{IEEE Transactions on Automatic Control}, vol.~55, no.~11, pp. 2576--2580, 2010.

\bibitem{rubagotti2014stabilizing}
M.~Rubagotti, P.~Patrinos, and A.~Bemporad, ``Stabilizing linear model predictive control under inexact numerical optimization,'' \emph{IEEE Trans. on Automatic Control}, vol.~59, no.~6, pp. 1660--1666, 2014.

\bibitem{karapetyan2023finite}
A.~Karapetyan, E.~C. Balta, A.~Iannelli, and J.~Lygeros, ``On the finite-time behavior of suboptimal linear model predictive control,'' in \emph{2023 62nd IEEE Conference on Decision and Control (CDC)}.\hskip 1em plus 0.5em minus 0.4em\relax IEEE, 2023, pp. 5053--5058.

\bibitem{karapetyan2025closed}
------, ``Closed-loop finite-time analysis of suboptimal online control,'' \emph{IEEE Transactions on Automatic Control}, 2025.

\bibitem{zanelli2021lyapunov}
A.~Zanelli, Q.~Tran-Dinh, and M.~Diehl, ``A {Lyapunov} function for the combined system-optimizer dynamics in inexact model predictive control,'' \emph{Automatica}, vol. 134, p. 109901, 2021.

\bibitem{chen2025sampled}
Y.~Chen, F.~Bullo, and E.~Dall'Anese, ``Sampled-data systems: Stability, contractivity and single-iteration suboptimal {MPC},'' \emph{arXiv preprint arXiv:2505.18336}, 2025.

\bibitem{abdelgalil2023multi}
M.~Abdelgalil, D.~E. Ochoa, and J.~I. Poveda, ``Multi-time scale control and optimization via averaging and singular perturbation theory: From odes to hybrid dynamical systems,'' \emph{Annual Reviews in Control}, vol.~56, p. 100926, 2023.

\bibitem{Andersson2019}
J.~A.~E. Andersson, J.~Gillis, G.~Horn, J.~B. Rawlings, and M.~Diehl, ``{CasADi} -- {A} software framework for nonlinear optimization and optimal control,'' \emph{Mathematical Programming Computation}, vol.~11, no.~1, pp. 1--36, 2019.

\end{thebibliography}
\end{document}